\documentclass[12pt]{article}
\usepackage{epsf}
\usepackage{amsbsy,amsmath}
\usepackage{mathtools}
\usepackage{mathrsfs}
\usepackage{amsfonts}
\usepackage{breqn}
\usepackage{amssymb}
\usepackage{enumitem}
\usepackage{eucal}
\usepackage{graphics,mathrsfs}
\usepackage{amsthm}
\usepackage{secdot}
\newtheorem{theorem}{Theorem}[section]
\newtheorem{lemma}[theorem]{Lemma}

\newtheorem{claim}[theorem]{claim}

\theoremstyle{definition}
\newtheorem{definition}[theorem]{Definition}

\theoremstyle{remark}

\numberwithin{equation}{section}

\numberwithin{equation}{section}

\newsavebox{\savepar}

\begin{document}

\title{Existence of multiple solutions to an elliptic problem with measure data}
\author{Amita Soni and D. Choudhuri}

\date{}
\maketitle

\begin{abstract}
\noindent In this paper we prove the existence of multiple nontrivial solutions of the following equation.
\begin{align*}
\begin{split}
-\Delta_{p}u & = \lambda |u|^{q-2}u+f(x,u)+\mu\,\,\mbox{in}\,\,\Omega,\\
u & = 0\,\, \mbox{on}\,\, \partial\Omega;
\end{split}
\end{align*}
where $\Omega \subset \mathbb{R}^N$ is a smooth bounded domain with $N \geq 3$, $1 < q^{\prime} < q < p-1;\\
\; \lambda,\;$ and $f$ satisfies certain conditions,  $\mu>0$ is a Radon measure, $q^{\prime}=\frac{q}{q-1}$ is the conjugate of $q$.
\end{abstract}
\begin{flushleft}

\end{flushleft}

\section{Introduction}
For many years now, the problem 
\begin{align}
\begin{split}
-\Delta_p u&=g(x,u),~\text{in}~\Omega\\
u&=0,~\text{on}~\partial\Omega,\label{p1}
\end{split}
\end{align}
has been studied extensively using the celebrated critical point theory which was introduced by Ambrosetti \& Rabinowitz in the Mountain pass theorem \cite{Ambro}. In order to apply the Mountain pass theorem one needs the Ambrosetti-Rabinowitz (AR) type condition on the nonlinear term $g$ which is as follows.\\
For $\theta>p$, $R>0$, we have
\begin{align}
\begin{split}
0&<\theta G(x,t)\leq g(x,t)t\label{c1}
\end{split}
\end{align}
$\forall |t|\geq R$ a.e. in $\Omega$, where $G(x,t)=\int_{0}^{t}g(x,s)ds$. The (AR) condition also implies that there exists positive constants $a$, $a_1$, $a_2$ such that $G(x,t)\geq a_1|t|^a-a_2$, $\forall (x,t)\in\Omega\times\mathbb{R}$. Thus $g$ is $p$-superlinear at infinity, in the sense that $\lim_{|t|\rightarrow\infty}\frac{G(x,t)}{|t|^p}=\infty$.\\
Of late, the problem in \eqref{p1} has been tackled without the AR condition by \cite{Costa, ittu1,kris1,Liu,li1,Miyagaki,sun1,wang1} and the references therein. Miyagaki \cite{Miyagaki} studied \eqref{p1} with a Laplacian by using the following condition on g: $\exists t_0>0$ such that $\frac{g(x,t)}{t}$ is increasing for $t\geq t_0$ and decreasing for $t\leq -t_0$ $\forall x\in\Omega$. The author in \cite{Miyagaki} guaranteed the existence of a nontrivial solution by using the Mountain Pass theorem with the Palais-Smale condition. Li et al \cite{li1} have extended this result, due to Miyagaki \cite{Miyagaki}, by replacing $-\Delta$ with $-\Delta_p$. In Li \cite{li1}, the authors needed the following subcritical growth condition $|g(x,t)|\leq C(1+|t|^{r-1})$ $\forall~t\in\mathbb{R}$ and for almost all $x$ in $\Omega$, $r\in[1,p^*)$, if $1<p<N$ and $p^*=\infty$ if $p\geq N$. A further generalized subcritical type growth condition was introduced by Lan \cite{lan1,lan2}, where $r=p^*$, to prove the existence of atleast one nontrivial weak solution to \eqref{p1} using the Mountain Pass theorem but without using the AR condition.\\
Motivated by the work due to Chung et al \cite{Chung} who have studied the existence of multiple solution for the problem
\begin{align}
\begin{split}
-\Delta_{p}u & = \lambda |u|^{q-2}u+f(x,u)\,\,\mbox{in}\,\,\Omega,\\
u & = 0\,\, \mbox{on}\,\, \partial\Omega;\label{mot1}
\end{split}
\end{align}
with concave-convex nonlinearities in bounded domains, we consider the following problem.
\begin{align}
\begin{split}
(P):~~-\Delta_{p}u & = \lambda |u|^{q-2}u+f(x,u)+\mu\,\,\mbox{in}\,\,\Omega,\\
u & = 0\,\, \mbox{on}\,\, \partial\Omega,\label{main_prob}
\end{split}
\end{align}
where $1<q'<q<p-1<p^*$, $p^{\ast}=\frac{Np}{N-p}$ is the Sobolev conjugate of $p$, $\mu>0$ is a Radon measure. 
We will prove the existence of multiple nontrivial weak solutions to the problem \eqref{main_prob}. 
The conditions we assume on the continuous function $f:\bar{\Omega}\times\mathbb{R}\rightarrow \mathbb{R}$ -  is slightly different from that assumed in \cite{Chung} - are as follows.

$(f_{0}) \underset{|t| \rightarrow \infty}{\text{lim}}\frac{f(x,t)}{{| t|}^{p^{\ast}-1}}= 0$ uniformly a.e. $x \in \Omega$.\\

$(f_{1})\;$ Let F be the primitive of $f$. There exists a positive constant $\bar{t} > 0$ such that \indent \indent $F(x,t) \geq 0$ a.e. $x \in \Omega$ and all $t \in [0,\bar{t}]$, where $F(x,t) = \int_0^{t}f(x,s)ds$.\\

$(f_{2}) \underset{|t| \rightarrow 0}{\text{lim sup}}\;\frac{F(x,t)}{{| t|}^{p}} < \lambda_{1}$ uniformly a.e. $x \in \Omega\;,\lambda_{1}$ being the first eigenvalue of $-\Delta_{p}$.\\

$(f_{3})  \underset{|t| \rightarrow \infty}{\text{lim}}\frac{F(x,t)}{{|t|}^{p}}= \infty$ uniformly a.e. $x \in \Omega$.\\

$(f_{4})\;$ There exists $\tilde{t} > 0$ such that for any $x \in \Omega$, the function $t\mapsto \frac{f(x,t)}{{\mid t\mid}^{p-2}t}$ is increasing \indent \indent if $t \geq \tilde{t}$ and decreasing if $t \leq -\tilde{t}\;, \forall\; x \in \Omega$.\\
We will denote the Sobolev space as $W_{0}^{1,p}(\Omega):=\left\lbrace u:\nabla u \in L^{p}(\Omega), u|_{\partial\Omega}=0\right\rbrace$ equipped with the norm ${\|.\|}_{1,p}$ which is defined as ${\|u\|}_{1,p}^{p}=\int_{\Omega}{| \nabla u|}^{p}dx$. We will denote ${\| .\|}_{1,p}$ as ${\| .\|}$ throughout this manuscript. We now state the main result of the paper which is as follows.
\begin{theorem}\label{Mainthm}
Suppose that $(f_{0})-(f_{4})$ hold. Then problem $(P)$ in \eqref{main_prob} possesses more than one nontrivial weak solution.
\end{theorem}

\section{Preliminary definitions}
We now discuss a few definitions, notations and essential results which will be used in this paper.
\begin{definition}
(Cerami condition) A functional $\Phi$ is said to satisfy the Cerami condition at a level $c \in \mathbb{R}$ if any sequence $(u_{n}) \subseteq X$ such that $\Phi(u_{n}) \rightarrow c$ and $(1+{\| u_{n}\|}){\Phi^{\prime}(u_{n})} \rightarrow 0$ has a convergent subsequence.
\end{definition}
\noindent In critical point theory, there are some situations in which a Palais-Smale sequence does not lead to a critical point, but a Cerami sequence can lead to a critical point. This whole thing based on the concept of `{\it linking}' (refer \cite {Martin}), for more details and examples. Cerami condition implies Palais-Smale condition and hence Cerami condition is a weaker condition than Palais-Smale. 
\begin{definition}\label{defn2}
	Let $(\mu_n)$ be a bounded sequence of measures in
	$\mathfrak{M}(\Omega)$. We say that $(\mu_n)$ converges to a measure $\mu \in \mathfrak{M}(\Omega)$ in the sense of measure if 
	$$\displaystyle{\int_{\Omega}\phi d\mu_n} \rightarrow
	\int_{\Omega}\phi d\mu\,\, \,\,\,\forall\,\,\phi\in
	C_0(\bar{\Omega}).$$ We denote this convergence by $\mu_n
	\xrightharpoonup{} \mu$. The topology defined via this weak
	convergence is metrizable and a bounded sequence with respect to
	this topology is pre-compact.
\end{definition}
\begin{definition}\label{defn3}
(Ekeland Variational Principle) Let $\Phi$ be a lower semicontinuous bounded below function from a Banach space X into $\mathbb{R} \cup \left\lbrace+\infty\right\rbrace$. For every $\epsilon > 0$, there is $x_{0} \in X$ such that $\Phi(x) \geq \Phi(x_{0})-\epsilon{\|x-x_{0}\|}$ for every $x \in X$ (refer {\cite{Ekeland}}).
\end{definition}
\noindent Throughout the article, we will denote the measure of a set $E$ in the sigma algebra of $\Omega$ as where $|E|$ and the absolute value of any real number, say $a$, as $|a|$.
\noindent We will use the Marcinkiewicz space  ${M}^q(\Omega)$ $\cite{Benilan}$ (or the weak $L^q(\Omega)$ space)  defined for every $0 < q <\infty$, as the space of all measurable functions $f:\Omega\rightarrow \mathbb{R}$ such that the corresponding distribution satisfy an estimate of the form
$$|\{x\in \Omega:|f(x)|>t\}|\leq \frac{C}{t^q},\hspace{0.4cm}t>0,C<\infty.$$ 
For bounded $\Omega$ we have ${M}^q\subset {M}^{\bar{q}}$ if $q\geq \bar{q}$, for some fixed positive $\bar{q}$. We recall here the following useful continuous embeddings
\begin{equation}\label{mar}
L^q(\Omega)\hookrightarrow {M}^q(\Omega)\hookrightarrow L^{q-\epsilon}(\Omega),
\end{equation}
for every $1<q<\infty$ and $0<\epsilon<q-1$.\\
We first consider a sequence of problems $(P_{n})$ which are as follows
\begin{align}
\begin{split}
-\Delta_{p}u & = \lambda |u|^{q-2}u+f(x,u)+\mu_{n}\,\,\mbox{in}\,\,\Omega,\\
u & = 0\,\, \mbox{on}\,\, \partial\Omega,
\end{split}
\end{align}
where $\mu_{n} \rightharpoonup \mu$ in measure. From here onwards we will denote $\int_{\Omega}fdx=\int_{\Omega}f$.
The corresponding energy functional of the sequence of problems $(P_{n})$ is written as
\begin{equation} 
I_{n}(u)= \frac{1}{p}\int_{\Omega}{|\nabla u|}^{p}dx - \frac{\lambda}{q}\int_{\Omega}{| u|}^{q}dx - \int_{\Omega}F(x,u) dx - \int_{\Omega}u d\mu_{n}
\end{equation}
and its Fr\'{e}chet derivative is defined as
\begin{equation}
<I_{n}^{\prime}(u),v> = \int_{\Omega}{|\nabla u|}^{p-2}\nabla u\nabla v dx - \lambda\int_{\Omega}{|u|}^{q-2}uv dx - \int_{\Omega}f(x,u)v dx - \int_{\Omega}\mu_{n}v dx
\end{equation}
$\forall u,v \in T$ where $T=W^{1, p}(\Omega)\cap C_0(\bar{\Omega})$, $C_0(\bar{\Omega})=\{\varphi\in C(\bar{\Omega}):\varphi|_{\partial\Omega}=0\}$ and $C(\bar{\Omega})$ will denote the space of continuous functions over $\bar{\Omega}$. We now define the corresponding energy functional of the problem $(P)$ as 
\begin{equation}
I(u)= \frac{1}{p}\int_{\Omega}{|\nabla u|}^{p}dx - \frac{\lambda}{q}\int_{\Omega}{|u|}^{q}dx - \int_{\Omega}F(x,u) dx - \int_{\Omega}u d\mu 
\end{equation}
and its Fr\'{e}chet derivative is defined as 
\begin{equation}
<I^{\prime}(u),v> = \int_{\Omega}{|\nabla u|}^{p-2}\nabla u\nabla v dx - \lambda\int_{\Omega}{|u|}^{q-2}uv dx - \int_{\Omega}f(x,u)v dx - \int_{\Omega}v d\mu
\end{equation}
for every $u,v\in T$.
\begin{definition}
$u\in S=\{u\in W_0^{1,s}(\Omega):||u||_{p^{*}}=1\}$, $s<\frac{N(p-1)}{N-1}$, is said to be a weak solution of the problem $(P)$ if
\begin{equation*}
\int_{\Omega}|\nabla u|^{p-2}\nabla u\nabla\varphi dxdy - \lambda\int_\Omega |u|^{q-2}u\varphi dx - \int_\Omega f(x,u)\varphi dx - \int_{\Omega}\varphi d\mu =0,
\end{equation*}
\end{definition}
$\forall$\; $\varphi \in T$.\\
 \section{Existence Results}
In order to prove the main result of this paper, given in the form of Theorem (1.1), we first prove a few lemmas related to the mountain pass theorem and the Cerami condition. We first develop the necessary tools for the mountain pass theorem. Observe that $I_{n}$'s are $C^{1}$ functionals defined over $W_{0}^{1,p}(\Omega)$. 

\begin{lemma}\label{lem1}
There exists $\lambda^{\prime}$ such that for all $\lambda \in (0,\lambda^{\prime})$, we can choose $\rho > 0, \eta > 0$ with $I_{n}(u) > \eta \; \forall \;u \in W_{0}^{1,p}(\Omega)$ and $\|u\| = \rho$.
\end{lemma}

\begin{proof}
From the assumption $(f2)$ we have, $\forall\; \epsilon > 0\; \exists\; \delta > 0$ such that $F(x,t)\leq (\lambda_{1}-\epsilon){|t|}^p, \;\forall |t| < \delta$. Hence,
\begin{align*}
I_{n}(u)&= \frac{1}{p}\int_{\Omega}{|\nabla u|}^{p}dx- \frac{\lambda}{q}\int_{\Omega}{|u|}^{q}dx - \int_{\Omega}F(x,u) dx - \int_{\Omega}u d\mu_{n}\\ 
&\geq \frac{1}{p}{\|u\|}^{p} - \frac{\lambda c_{1}}{q}{\|u\|}^{q} - \frac{(\lambda_{1}-\epsilon)}{p}{\|u\|}_{p}^{p}- c_{2}{\|u\|}_{p}{\|\mu_{n}\|}_{p^{\prime}}\\ &\geq \frac{1}{p}{\|u\|}^{p} - \frac{\lambda c_{1}}{q}{\|u\|}^{q} - \frac{(\lambda_{1}-\epsilon)}{p\lambda_{1}}{\|u\|}^{p}- c_{2}{\|u\|}^{p^{\ast}}{\|\mu_{n}\|}_{p^{\prime}}\\ &= \left[\frac{\epsilon}{p\lambda_{1}} - \left\lbrace \frac{\lambda c_{1}}{q}{\|u\|}^{q-p} + c_{2}{\|u\|}^{p^{\ast}-p}{\|\mu_{n}\|}_{p^{\prime}}\right\rbrace \right] {\|u\|}^{p}, 
\end{align*}
where we have used the Rayleigh constant $\lambda_{1}=\underset{\underset{u\neq 0}{u\in {W_{0}^{1,p}}\left(\Omega\right)}}{\text{min}}\left\lbrace \frac{\int_{\Omega}{|\nabla u|}^{p}dx}{\int_{\Omega}{|u|}^{p}dx}\right\rbrace $. Consider a continuous function $\tau_{\lambda} : (0,\infty) \rightarrow \mathbb{R}$ defined as $\tau_{\lambda}(t)=  \frac{\lambda c_{1}}{q}{|t|}^{q-p} + c_{2}{\|\mu_{n}\|}_{p^{\prime}}{|t|}^{p^*-p}$.
\noindent Since, $1 < q < p < p^{\ast}$, so it can be seen that $\underset{t\rightarrow \infty}{\text{lim}}\tau_{\lambda}(t)= \underset{t \rightarrow 0^{+}}{\text{lim}}\tau_{\lambda}(t)= +\infty$.
Hence, it is possible to find a `$t_{\ast}$' such that $0 < \tau_{\lambda}(t_{\ast})= \underset{t \in (0,\infty)}{\text{min}} \tau_{\lambda}(t)$.
On solving for $t_{\ast}$ such that $\tau'_{\lambda}(t_*)= 0$, we get $t_{\ast} = {\left[ \frac{\lambda c_{1}(p-q)}{q c_{2}(p^{\ast}-p){\|\mu_{n}\|}_{p^{\prime}}}\right]}^{\frac{1}{p^{\ast}-q}}$. This implies   
\begin{align}
\tau_{\lambda}(t_{\ast})&= k\lambda^{\frac{p^{\ast}-p}{p^{\ast}-q}} \rightarrow 0~\text{as}~\lambda \rightarrow 0.\label{min1}
\end{align}
Thus, by choosing $\|u\| = \rho$ and from \eqref{min1} there exists a $\lambda^{\prime}$ such that $\forall \lambda\in (0,\lambda')$ we have $I_n(u)>\eta$.
\end{proof}
\noindent The following lemma guarantees the existence of a function $v$ such that $I_{n}(v) < 0$.
\begin{lemma}\label{lem2}
There exists $e_{1} \in W_{0}^{1,p}(\Omega),  {\| e_{1}\|} > 0$, such that $I_{n}(te_{1}) < 0 $ for sufficiently large $t$.
\end{lemma}
\begin{proof}
Let $e_{1} \in W_{0}^{1,p}(\Omega)$ with ${\|e_{1}\|} > 0$. From the assumption $(f3)$, we have $\forall~M > 0,\; \exists\; k(M) > 0$ such that $F(x,t) \geq M{|t|}^{p}- k(M)$ a.e. in $\Omega,\; \forall t \in \mathbb{R}$.
\begin{align*}
I_{n}(te_{1})&=\frac{1}{p}\int_{\Omega}{|\nabla(te_{1})|}^{p}dx - \frac{\lambda}{q}\int_{\Omega}{|te_{1}|}^{q}dx - \int_{\Omega}F(x,te_{1})dx - \int_{\Omega}(te_{1})\mu_{n}dx\\
&\leq \frac{{|t|}^p}{p}{\|e_{1}\|}^{p}-\frac{\lambda {{|t|}^{q}}}{q}\int_{\Omega}{|e_{1}|}^{q}dx - M{{|t|}^{p}}\int_{\Omega}{|e_{1}|}^{p}dx+k(M){|\Omega|}- t\int_{\Omega}e_{1}\mu_{n}dx
\end{align*}
Choose $M$ large enough such that the whole quantity becomes negative. Hence, $I_{n}(te_{1}) < 0$ for $t$ sufficiently large.
\end{proof}

\begin{lemma}\label{lem3}
There exists $e_{2} \in W_{0}^{1,p}(\Omega), \|e_{2}\| > 0$ such that $I_{n}(te_{2}) < 0;\, \forall t > 0$ in a small neighborhood of 0.
\end{lemma}
\begin{proof}
Let $e_{2} \in W_{0}^{1,p}(\Omega)$ with ${\|e_{2}\|} > 0$. From the assumption in $(f_{1})$, we have, $\;\forall t \in [0,\bar{t}]$ and $x \in \Omega$ a.e., $F(x,t) \geq 0$. So, for $t \in \left(0,\frac{\bar{t}}{\|e_{2}\|}_{L^{\infty}(\Omega)}\right)$, we have
\begin{align*}
I_{n}(te_{2})&= \frac{1}{p}\int_{\Omega}{|\nabla(te_{2})|}^{p}dx - \frac{\lambda}{q}\int_{\Omega}{|te_{2}|}^{q}dx - \int_{\Omega}F(x,te_{2})dx - \int_{\Omega}(te_{2})\mu_{n}dx\\
&\leq \frac{t^p}{p}{\|e_{2}\|}^{p}-\frac{\lambda {t^q}}{q}\int_{\Omega}{|e_{2}|}^{q}dx-t\int_{\Omega}e_{2}\mu_{n}dx\\ 
&\leq \frac{t^{p}}{p}{\|e_{2}\|}^{p}-\frac{\lambda {t^q}}{q}\int_{\Omega}{|e_{2}|}^{q}dx.
\end{align*}
We need to find a $t > 0$ for which $I_{n}(te_{2})$ is less than 0. For this we consider,
\begin{align*}
0&>\frac{t^{p}}{p}{\|e_{2}\|}^{p}-\frac{\lambda {t^q}}{q}\int_{\Omega}{| e_{2}|}^{q}dx \\ 
&= t^{q}\left[\frac{t^{p-q}}{p}{\|e_{2}\|}^{p}-\frac{\lambda}{q}\int_{\Omega}{|e_{2}|}^{q}dx\right].
\end{align*}
Thus,
\begin{align*}
\frac{t^{p-q}}{p}{\|e_{2}\|}^{p} &< \frac{\lambda}{q}\int_{\Omega}{|e_{2}|}^{q}dx
\end{align*}
and so
\begin{align*}
t& < {\left(\frac{\lambda p\int_{\Omega}{|e_{2}|}^{q}dx}{q{\|e_{2}\|}^{p}}\right)}^{\frac{1}{p-q}}.
\end{align*}

Thus, if $0 < t < \min\left\lbrace {\left(\frac{\lambda p\int_{\Omega}{|e_{2}|}^{q}dx}{q{\|e_{2}\|}^{p}}\right)}^{\frac{1}{p-q}},\frac{\bar{t}}{{\|e_{2}\|}_{L^{\infty}(\Omega)}}\right\rbrace$ then $I_{n}(te_{2}) < 0$.
\end{proof}

\begin{lemma}\label{lem4}
The functional $I_n$ satisfies the Cerami condition.
\end{lemma}

\begin{proof}
Let $(u_{m,n})$ be a sequence in $W_{0}^{1,p}(\Omega)$ such that $I_{n}(u_{m,n}) \to c$, $
(1+{\| u_{m,n}\|}){\|I_{n}^{\prime}(u_{m,n})\|}\rightarrow 0$ as $m\rightarrow \infty$, where ${\|I_{n}^{\prime}(u_{m,n})\|}=\mathrm{sup}\;\left\lbrace \mid<I_{n}^{\prime}(u_{m,n}),\phi>\mid : \phi \in W_{0}^{1,p}(\Omega), {\|\phi\|}=1\right\rbrace$. We first show that $(u_{m,n})$ is bounded. For if not, i.e. ${\| u_{m,n}\|} \rightarrow \infty$ as $m \rightarrow \infty$, define $v_{m,n}=\frac{u_{m,n}}{\|u_{m,n}\|}$ so that ${\|v_{m,n}\|}=1$. Since, $W_{0}^{1,p}(\Omega)$ is a reflexive space, so $(v_{m,n})$ has a weakly convergent subsequence in $W_{0}^{1,p}(\Omega)$. Let $v_{m,n} \rightharpoonup v_{0}$ in $W_{0}^{1,p}(\Omega).$ Due to the compact embedding we have
\begin{align}
v_{m,n} \rightarrow v_{0}\; \mathrm{in}\; L^{r}(\Omega)\;\mathrm{for}\; r \in\; [1,p^{\ast})\mathrm{\;and\; hence\; upto\; a\; subsequence}\\ \MoveEqLeft \hspace{-106mm}v_{m,n}(x) \rightarrow v_{0}(x)\;\; \mathrm{a.e.\; in}\; \Omega \;\mathrm{as}\; m \rightarrow \infty.
\end{align}
We now have two cases.\\
Case (i): When $v_{0} \neq 0$.\\
Let $\Omega^{\prime}= \left\lbrace x \in \Omega : v_{0}(x) \neq 0\right\rbrace$. If $x \in \Omega^{\prime}$, then 
\begin{equation}
\hspace{-70mm}|u_{m,n}(x)|= |v_{m,n}(x)| {\|u_{m,n}\|} \rightarrow \infty \;\mathrm{a.e.\; in}\; \Omega^{\prime}.
\end{equation}
Since, $I_{n}(u_{m,n}) \rightarrow c$, we have $ \frac{I_{n}(u_{m,n})}{{\|u_{m,n}\|}^{p}} \rightarrow 0$. Hence, as $m \to \infty$
\begin{align*}
o(1)=\frac{1}{p}-\frac{\lambda}{q}{\int_{\Omega}\frac{{|u_{m,n}|}^q}{{\|u_{m,n}\|}^{p}}}dx - {\int_{\Omega^{\prime}}\frac{F(x,u_{m,n})}{{\|u_{m,n}\|}^{p}}dx}- {\int_{\Omega \setminus\Omega^{\prime}}\frac{F(x,u_{m,n})}{{\| u_{m,n}\|}^{p}}dx} - \int_{\Omega}\frac{u_{m,n}\mu_{n}}{{\|u_{m,n}\|}^{p}}dx .
\end{align*}
Using the Rayleigh constant $\lambda_{1}=\underset{\underset{u_{m,n}\neq 0}{u_{m,n}\in {W_{0}^{1,p}}\left(\Omega\right)}}{\text{min}}\left\lbrace \frac{\int_{\Omega}{|\nabla u_{m,n}|}^{p}dx}{\int_{\Omega}{|u_{m,n}|}^{p}dx}\right\rbrace $, we get 
$\int_{\Omega}{|u_{m,n}|}^{p} \leq \frac{{\|u_{m,n}\|}^p}{\lambda_{1}} 
$. This implies that $c\int_{\Omega}{|u_{m,n}|}^{q} \leq \frac{{\|u_{m,n}\|}^p}{\lambda_{1}}$, since $q < p$.\\

Thus,
\begin{align}
\begin{split}
o(1)&\leq \frac{1}{q}-\frac{\lambda}{q}{\int_{\Omega}\frac{{|u_{m,n}|}^q}{{\|u_{m,n}\|}^{p}}}dx - {\int_{\Omega^{\prime}}\frac{F(x,u_{m,n})}{{\|u_{m,n}\|}^{p}}dx}- {\int_{\Omega \setminus\Omega^{\prime}}\frac{F(x,u_{m,n})}{{\|u_{m,n}\|}^{p}}dx} - \int_{\Omega}\frac{u_{m,n}\mu_{n}}{{\| u_{m,n}\|}^{p}}dx \\
&\leq \frac{1}{q}\mathrm{max}\left\lbrace 1,1-\frac{\lambda}{c\lambda_{1}}\right\rbrace - {\int_{\Omega^{\prime}}\frac{F(x,u_{m,n})}{{\|u_{m,n}\|}^{p}}dx}- {\int_{\Omega \setminus\Omega^{\prime}}\frac{F(x,u_{m,n})}{{\|u_{m,n}\|}^{p}}dx} - \frac{\int_{\Omega}u_{m,n}\mu_{n}}{{\|u_{m,n}\|}^{p}}dx \\
&\leq \frac{1}{q}\mathrm{max}\left\lbrace 1,1-\frac{\lambda}{c\lambda_{1}}\right\rbrace - {\int_{\Omega^{\prime}}\frac{F(x,u_{m,n})}{{\|u_{m,n}\|}^{p}}dx}- {\int_{\Omega \setminus\Omega^{\prime}}\frac{F(x,u_{m,n})}{{\|u_{m,n}\|}^{p}}dx} +\left|\frac{\int_{\Omega}u_{m,n}\mu_{n}}{{\|u_{m,n}\|}^{p}}dx \right|\\
&\leq \frac{1}{q}\mathrm{max}\left\lbrace 1,1-\frac{\lambda}{c\lambda_{1}}\right\rbrace - {\int_{\Omega^{\prime}}\frac{F(x,u_{m,n})}{{\|u_{m,n}\|}^{p}}dx}- {\int_{\Omega \setminus\Omega^{\prime}}\frac{F(x,u_{m,n})}{{\|u_{m,n}\|}^{p}}dx} + \frac{{\|\mu_{n}\|}_{p^{\prime}}{\|u_{m,n}\|}_{p}}{{\|u_{m,n}\|}^{p}}\\
&\leq \frac{1}{q}\mathrm{max}\left\lbrace 1,1-\frac{\lambda}{c\lambda_{1}}\right\rbrace - {\int_{\Omega^{\prime}}\frac{F(x,u_{m,n})}{{\|u_{m,n}\|}^{p}}dx}- {\int_{\Omega \setminus\Omega^{\prime}}\frac{F(x,u_{m,n})}{{\|u_{m,n}\|}^{p}}dx} + \frac{{\|\mu_{n}\|}_{p^{\prime}}{\|u_{m,n}\|}(\lambda_{1})^{\frac{-1}{p}}}{{\|u_{m,n}\|}^{p}}\label{ineq1}
\end{split}
\end{align}
Also, \begin{align*}
\begin{split}
\frac{F(x,u_{m,n})}{{\|u_{m,n}\|}^{p}}&= \frac{F(x,u_{m,n})}{{|u_{m,n}(x)|}^{p}}.\frac{{|u_{m,n}(x)|}^{p}}{{\|u_{m,n}\|}^{p}}\\&= \frac{F(x,u_{m,n})}{{|u_{m,n}(x)|}^{p}}.{|v_{m,n}(x)|}^{p}
\end{split}
\end{align*}
Since
$\underset{{|t|} \rightarrow \infty}{\text{lim}}\frac{F(x,t)}{{|t|}^{p}}=\infty$ and $v_{m,n} \rightarrow v_{0}$ in $L^{p}(\Omega)$ with $v_{0}(x)\neq 0$, then
$\frac{F(x,u_{m,n})}{{\|u_{m,n}\|}^{p}}\rightarrow \infty$ a.e. in $\Omega^{\prime}$. Using the Fatou's lemma, we have $\underset{m \to \infty}{\text{lim}}\int_{\Omega^{\prime}}\frac{F(x,u_{m,n})}{{\|u_{m,n}\|}^{p}}dx=\infty$.\\
From the assumption in $(f_{3}), \underset{|t| \rightarrow \infty}{\text{lim}}F(x,t) = \infty$ uniformly in $\bar{\Omega}$ and hence, there exists two positive constants $\bar{t}$ and M such that $F(x,t) \geq M$ for every $x \in \bar{\Omega}$ and for all t such that $|t| > \bar{t}$. Since $F$ is continuous on $\bar{\Omega}\times\mathbb{R}$, so $|F(x,t)| \leq  c_1$ for every $
 x \in \bar{\Omega}$ and $|t| \leq \bar{t}$.
Therefore, there exists a $k$ such that
\begin{align}
F(x,t) \geq k\; \mathrm{for\; any}\;(x,t) \in\bar{\Omega}\times\mathbb{R}.
\end{align}
By our assumption that ${\|u_{m,n}\|}$ is unbounded in $W_{0}^{1,p}(\Omega)$ and using (3.5),\\
$\underset{m \rightarrow \infty}{\text{lim}}\int_{\Omega \setminus\Omega^{\prime}}\frac{F(x,u_{m,n})}{{\|u_{m,n}\|}^{p}}dx \geq \underset{m \rightarrow \infty}{\text{lim}}\frac{k|\Omega\setminus\Omega'|}{\|u_{m,n}\|^p}=0$.
The last term in \eqref{ineq1} converges to $0$ owing to $p > 1$. This yields a contradiction that $0\leq -\infty$. Hence, ${\|u_{m,n}\|}$ is bounded in $W_{0}^{1,p}(\Omega)$.\\\\
Case(ii): When $v_{0}= 0$.\\
Since, $t\mapsto I_{n}(tu_{m,n})$ is continuous in $t \in [0,1]$, hence for each $m$ there exists $t_{m} \in [0,1]$ such that $I_{n}(t_{m}u_{m,n})=\underset{t\in[0,1]}{\text{max}}I_{n}(tu_{m,n})$.
For any $k\in\mathbb{N},$ choose  $r_{k,n}=\left(2p{\|u_{l,n}\|}^{p}\right)^{\frac{1}{p}}$ such that $r_{k,n}{\|u_{m,n}\|}^{-1} \in (0,1)$ for any fixed big integer $k$.
Using the dominated convergence theorem and the fact that $v_{0}=0$, we get $\underset{m \rightarrow \infty}{\text{lim}}\int_{\Omega}{|r_{k,n}v_{m,n}(x)|}^{q}dx=0$ and $\underset{m \rightarrow \infty}{\text{lim}}\int_{\Omega}{|r_{k,n}v_{m,n}(x)|}^{p}dx=0$. Since, $v_{m,n}(x) \rightarrow v_{0}(x)$ a.e. $\Omega$ and $F$ is continuous so
$F(x,r_{k,n}v_{m,n}(x)) \rightarrow F(x,r_{k,n}v_{0}(x))$ a.e. in $\Omega$.\\
From $(f_{0})\;\forall \epsilon > 0,\; \exists\; c(\epsilon) > 0$ such that $|F(x,t)| \leq \frac{\epsilon}{c_{1}}{|t|}^{p^{\ast}}+c(\epsilon),\;\forall\; t \in \mathbb{R}, \mathrm{a.e.\; in}\; \Omega$. Using the dominated convergence theorem, $\int_{\Omega}F(x,r_{k,n}v_{m,n}(x))\rightarrow 0$ as $ m \rightarrow \infty,\; \forall\; k \in \mathbb{N}$ since $F(x,0)=0$.
\begin{align*}
I_{n}(t_{m}u_{m,n})\geq& I_{n}(r_{k,n}{\|u_{m,n}\|}^{-1}u_{m,n})\\=&I_{n}(r_{k,n}v_{m,n})\\=&
\frac{1}{p}\int_{\Omega}{|\nabla r_{k,n}v_{m,n}|}^{p} dx-\frac{\lambda}{q}\int_{\Omega}{|r_{k,n}v_{m,n}|}^{q} dx-\int_{\Omega}F(x,r_{k,n}v_{m,n})dx
-\int_{\Omega}r_{k,n}v_{m,n}\mu_{n}dx&\\
\geq& \frac{1}{p}\int_{\Omega}\left({\|u_{k,n}\|}^{p}(2p){|\nabla v_{m,n}|}^{p}\right) dx-\frac{\lambda}{q}\int_{\Omega}{|r_{k,n}v_{m,n}|}^{q} dx-\int_{\Omega}F(x,r_{k,n}v_{m,n})dx\\-&\int_{\Omega}{|r_{k,n}v_{m,n}\mu_{n}|} dx&\\
\geq& 2{\|u_{k,n}\|}^{p}{\|v_{m,n}\|}^{p}-\frac{\lambda}{q}\int_{\Omega}{|r_{k,n}v_{m,n}|}^{q} dx-\int_{\Omega}F(x,r_{k,n}v_{m,n})dx-{\|\mu_{n}\|}_{p^{\prime}}{\|r_{k,n}v_{m,n}\|}_{p} 
\end{align*}
Since the last three term tends to zero as $n \to \infty$ so 
\begin{align}\label{eq3.7}
I_{n}(t_{m}u_{m,n}) \geq {\|u_{k,n}\|}^{p}
\end{align}
As $\| u_{k,n}\| \rightarrow \infty$ as $k \rightarrow \infty$ so $I_{n}(t_{m}u_{m,n}) \rightarrow \infty$ as $m \rightarrow \infty$ for any large integer $k$.
Since $I_{n}(u_{m,n}) \rightarrow c $ and $I_{n}(0)=0$. So, for $t_{m} \in (0,1), I_{n}^{\prime}(t_{m}u_{m,n})=0$ for any $n \in \mathbb{N}$ and
$\langle I_{n}^{\prime}(t_{m}u_{m,n}),t_{m}u_{m,n}\rangle = t_{m}\frac{d}{dt}\mid_{t=t_{m}}I_{n}(tu_{m,n})=0$.\\
\begin{align*}
I_{n}(t_{m}u_{m,n})=&I_{n}(t_{m}u_{m,n})-\frac{1}{p}\left\langle I_{n}^{\prime}(t_{m}u_{m,n}),t_{m}u_{m,n}\right\rangle  \\
=&\frac{1}{p}\int_{\Omega}{|\nabla t_{m}u_{m,n}|}^{p} dx-\frac{\lambda}{q}\int_{\Omega}{|t_{m}u_{m,n}|}^{q} dx-\int_{\Omega}F(x,t_{m}u_{m,n})dx-\int_{\Omega}t_{m}u_{m,n}\mu_{n}dx\\
&-\left\{\frac{1}{p}\int_{\Omega}{|\nabla t_{m}u_{m,n}|}^{p} dx\right.\\
&-\left.\frac{\lambda}{p}\int_{\Omega}{|t_{m}u_{m,n}|}^{q} dx-\frac{1}{p}\int_{\Omega}f(x,t_{m}u_{m,n})(t_{m}u_{m,n})dx-\frac{1}{p}\int_{\Omega}t_{m}u_{m,n}\mu_{n}dx\right\}\\
=&\lambda\left(\frac{1}{p}-\frac{1}{q}\right)\int_{\Omega}{|t_{m}u_{m,n}|}^{q}dx+\frac{1}{p}\int_{\Omega}f(x,t_{m}u_{m,n})(t_{m}u_{m,n})dx-\int_{\Omega}F(x,t_{m}u_{m,n})dx\\
&-\left(1-\frac{1}{p}\right)\int_{\Omega}t_{m}u_{m,n}\mu_{n}dx
\end{align*}
This implies
\begin{align*}I_{n}(t_{m}u_{m,n})+A\int_{\Omega}t_{m}u_{m,n}\mu_{n}dx &\leq \frac{1}{p}\int_{\Omega}f(x,t_{m}u_{m,n})(t_{m}u_{m,n})dx-\int_{\Omega}F(x,t_{m}u_{m,n})dx\\
&=\frac{1}{p}\int_{\Omega}\tilde{F}(x,t_{m}u_{m,n})dx,
\end{align*}
where $A= \left(1-\frac{1}{p}\right)$.\\ 
Using the Lemma 2.3 from \cite{Liu}, which states that
\begin{lemma}\label{liu}
If ($f_4$) holds, then for any $x\in \Omega$, $\tilde{F}(x,t)$ is increasing in $t\geq\bar{t}$ and decreasing in $t \leq -\bar{t}$, where $\tilde{F}(x,t)=f(x,t)t-pF(x,t)$. In particular, there exists $C_1 > 0$ such that $\tilde{F}(x,s)\leq \tilde{F}(x,t)+C_1$ for $x\in\Omega$ and $0 \leq s \leq t$ or $t \leq s \leq 0$,
\end{lemma}
\noindent we get
\begin{align*}
\frac{1}{p}\int_{\Omega}\tilde{F}(x,t_{m}u_{m,n})dx &= \frac{1}{p}\int_{\left\lbrace u_{m,n}\geq 0\right\rbrace}\tilde{F}(x,t_{m}u_{m,n})dx + \frac{1}{p}\int_{\left\lbrace u_{m,n} < 0\right\rbrace}\tilde{F}(x,t_{m}u_{m,n})dx\\
&\leq\frac{1}{p}\int_{\left\lbrace u_{m,n}\geq 0\right\rbrace}\left[\tilde{F}(x,u_{m,n})+c_{1}\right]dx +\frac{1}{p}\int_{\left\lbrace u_{m,n}< 0\right\rbrace}\left[\tilde{F}(x,u_{m,n})+c_{1}\right]dx\\ &= \frac{1}{p}\int_{\Omega}\tilde{F}(x,u_{m,n})dx+\frac{1}{p}c_1|\Omega|\\
&=I_{n}(u_{m,n})-\frac{1}{p}<I_{n}^{\prime}(u_{m,n}),u_{m,n}> + \lambda\left(\frac{1}{q}-\frac{1}{p}\right)\int_{\Omega}{|u_{m,n}|}^{q}dx\\
&~~~+A\int_{\Omega}u_{m,n}\mu_{n}dx+\frac{1}{p}c_1{|\Omega|}
\end{align*}
\begin{align*}
&\leq I_{n}(u_{m,n})-\frac{1}{p}<I_{n}^{\prime}(u_{m,n}),u_{m,n}> + \lambda\left(\frac{1}{q}-\frac{1}{p}\right)\int_{\Omega}{|u_{m,n}|}^{q}dx\\
&~~~+A\int_{\Omega}|{u_{m,n}\mu_{n}|}dx+\frac{1}{p}c_{1}{|\Omega|}\\
&\leq I_{n}(u_{m,n})-\frac{1}{p}<I_{n}^{\prime}(u_{m,n}),u_{m,n}> + \lambda\left(\frac{1}{q}-\frac{1}{p}\right)\int_{\Omega}{|u_{m,n}|}^{q}dx\\
&~~~+A{\|\mu_{n}\|}_{p^{\prime}}{\|u_{m,n}\|}+\frac{1}{p}c_{1}{|\Omega|}\\
&\leq I_{n}(u_{m,n})-\frac{1}{p}\left\langle I_{n}^{\prime}(u_{m,n}),u_{m,n}\right\rangle + \lambda\left(\frac{1}{q}-\frac{1}{p}\right)\int_{\Omega}{|u_{m,n}|}^{q}dx\\
&~~~+A{\|\mu_{n}\|}_{p^{\prime}}{\|u_{m,n}\|}^{q}+\frac{1}{p}c_{1}{|\Omega|}
\\
&\leq I_{n}(u_{m,n})-\frac{1}{p}\langle I_{n}^{\prime}(u_{m,n}),u_{m,n}\rangle + c_0\lambda\left(\frac{1}{q}-\frac{1}{p}\right){\|u_{m,n}\|}^{q}dx\\
&~~~+A{\|\mu_{n}\|}_{p^{\prime}}{\|u_{m,n}\|}^{q}+\frac{1}{p}c_{1}{|\Omega|}
\end{align*}
Hence
\begin{align*}
I_{n}(t_{m}u_{m,n})+A\int_{\Omega}t_{m}u_{m,n}\mu_{n}dx&\leq I_{n}(u_{m,n})-\frac{1}{p}\langle I_{n}^{\prime}(u_{m,n}),u_{m,n}\rangle\\ &+c_0\lambda\left(\frac{1}{q}-\frac{1}{p}\right){\|u_{m,n}\|}^{q}dx+A{\|\mu_{n}\|}_{p^{\prime}}{\|u_{m,n}\|}^{q}+\frac{1}{p}c_{1}{|\Omega|}.
\end{align*}
This implies
\begin{align*}
 I_{n}(t_{m}u_{m,n})\leq I_{n}(u_{m,n})-\frac{1}{p}\left\langle I_{n}^{\prime}(u_{m,n}),u_{m,n}\right\rangle +c_0\lambda\left(\frac{1}{q}-\frac{1}{p}\right){\| u_{m,n}\|}^{q}dx\\ \MoveEqLeft \hspace*{-80mm}+A{\|\mu_{n}\|}_{p^{\prime}}{\|u_{m,n}\|}^{q}+\frac{1}{p}c_{1}{|\Omega|}-A\int_{\Omega}t_{m}u_{m,n}\mu_{n}dx
\end{align*}
Using \eqref{eq3.7}, we get
\begin{align}
\begin{split}
{\|u_{k,n}\|}^{p}&\leq I_{n}(u_{m,n})-\frac{1}{p}<I_{n}^{\prime}(u_{m,n}),u_{m,n}>+c_0\lambda\left(\frac{1}{q}-\frac{1}{p}\right){\|u_{m,n}\|}^{q}dx\\ &+A{\|\mu_{n}\|}_{p^{\prime}}{\|u_{m,n}\|}^{q}+\frac{c_{1}}{p}{|\Omega|}-A\int_{\Omega}t_{m}u_{m,n}\mu_{n}dx\label{eq11}
\end{split}
\end{align}
Now, $\left|\frac{A\int_{\Omega}t_{m}u_{m,n}\mu_{n}dx}{{\|u_{k,n}\|}^{p}}\right| \leq \frac{A t_{m}\int_{\Omega}|u_{m,n}\mu_{n}| dx}{{\|u_{k,n}\|}^{p}} \leq  \frac{A t_{m}{\|u_{m,n}\|}_{p}{\|\mu_{n}\|}_{p^{\prime}}}{{\|u_{k,n}\|}^{p}} \leq \frac{A t_{m}{\|u_{m,n}\|}{\|\mu_{n}\|}_{p^{\prime}}}{{\|u_{k,n}\|}^{p}} \rightarrow 0\;$
as $\| u_{k,n}\| \rightarrow \infty$ $\forall k\geq m$.\\
Since $q<p-1$, hence on dividing \eqref{eq11} by ${\|u_{k,n}\|}^{p}$ and letting ${\|u_{k,n}\|}^{p} \rightarrow \infty$, as $k\rightarrow\infty$, we get $1 \leq 0$ which is a contradiction. Hence, $(u_{m,n})$ is bounded in $W_{0}^{1,p}(\Omega)$.\\
The next step is to show that $(u_{m,n})$ admits a strongly convergent subsequence in $W_0^{1,p}(\Omega)$. Since $W_{0}^{1,p}(\Omega)$ is a reflexive space so $(u_{m,n})$ has a subsequence which converges weakly to $u_n$ in $W_{0}^{1,p}(\Omega)$ and strongly in $L^{r}(\Omega)$ for $r \in [1,p^{\ast})$ due to Rellich's compact embedding. We also  have that ${\|u_{m,n}\|}_{p^{\ast}}^{p^{\ast}} \leq c_{2}$. 
\begin{equation} \label{eq3.8}
\begin{split}
\langle I_{n}^{\prime}(u_{m,n}),u_{m,n}-u_n\rangle & = \int_{\Omega}{|\nabla u_{m,n}|}^{p-2}\nabla u_{m,n}\cdot(\nabla u_{m,n}-\nabla u)dx-\int_{\Omega}{| u_{m,n}|}^{q-2}u_{m,n}(u_{m,n}-u_n)dx \\
 &  -\int_{\Omega}f(x,u_{m,n})(u_{m,n}-u_n)-\int_{\Omega}\mu_{n}(u_{m,n}-u_n)dx
\end{split}
\end{equation}
From the assumption in $(f_{0})$, we have $\forall\; \epsilon > 0 \;\exists\;m(\epsilon) > 0$ such that $|f(x,t)t| \leq \frac{\epsilon}{2c_{2}}{|t|}^{p^{\ast}}+m(\epsilon),\;\forall\; t \in \mathbb{R}, \mathrm{a.e. in}\; \Omega$. Choose $\delta=\frac{\epsilon}{2m(\epsilon)} > 0, F \subseteq \Omega$ such that $\mu(F) < \delta$ then
\begin{align}
\begin{split}
\left|\int_{F}f(x,u_{m,n})u_{m,n}dx\right| &\leq \int_{F}\left|f(x,u_{m,n})u_{m,n}\right|dx \\&\leq \int_{F}m(\epsilon)dx+\frac{\epsilon}{2c_{2}}\int_{F}{|u_{m,n}|}^{p^{\ast}}dx \\&\leq \epsilon\label{eq12}
\end{split}
\end{align}
Hence, $\left\lbrace \int_{\Omega}f(x,u_{m,n})u_{m,n}dx:m \in \mathbb{N}\right\rbrace$ is equiabsolutely continuous and therefore from the Vitali convergence theorem we get
\begin{equation}
\int_{\Omega}f(x,u_{m,n})u_{m,n}dx \rightarrow \int_{\Omega}f(x,u_n)u_n\; dx\; \mathrm{as} \;m \rightarrow \infty.\label{eq13}
\end{equation}
Based on exactly the same line of argument using the assumption in ($f_0$), 
it can be shown that, $\left\lbrace \int_{\Omega}f(x,u_{m,n})u_n\;dx:m \in \mathbb{N}\right\rbrace$ is equiabsolutely continuous and therefore from the Vitali convergence theorem we get
\begin{equation}
\int_{\Omega}f(x,u_{m,n})u_n\;dx \rightarrow \int_{\Omega}f(x,u_n)u_n\; dx\; \mathrm{as} \;m \rightarrow \infty\label{eqemer1}
\end{equation}
From ($\ref{eq13}$) and ($\ref{eqemer1}$), we get
\begin{align}
\int_{\Omega}f(x,u_{m,n})(u_{m,n}-u_n)dx \rightarrow 0 \; \mathrm{as} \;m \rightarrow \infty\label{eqemer2}
\end{align}
Again by using the H\"{o}lder's inequality and compact embedding results, we have 
\begin{align}
\int_{\Omega}{|u_{m,n}|}^{q-2}u_{m,n}(u_{m,n}-u_n)dx & \leq \int_{\Omega}\mid{|u_{m,n}|}^{q-2}u_{m,n}(u_{m,n}-u_n)\mid dx \nonumber\\
&= \int_{\Omega}{|u_{m,n}|}^{q-1}{|u_{m,n}-u_n|}dx\nonumber\\
&\leq \left(\int_{\Omega}{|u_{m,n}|}^{q}dx\right)^{\frac{q-1}{q}}\left(\int_{\Omega}{|u_{m,n}-u_n|}^{q}dx\right)^{\frac{1}{q}} \rightarrow 0 \;\mathrm{as}\; m \rightarrow \infty\label{eqemer3}
\end{align}
Since, $u_{m,n} \rightarrow u_n$ in $L^{p}(\Omega)$, so 
\begin{align}
\begin{split}
\int_{\Omega}(u_{m,n}-u_n)\mu_{n} &\leq {\|u_{m,n}-u_n\|}_{p}{\|\mu_{n}\|}_{p^{\prime}} \rightarrow 0~\text{as}~m \rightarrow \infty.\label{eq14}
\end{split}
\end{align}
We know that, $\left\langle I_{n}^{\prime}(u_{m,n}),u_{m,n}-u_n\right\rangle  \rightarrow 0$ as $m \rightarrow \infty$. Hence, from \eqref{eqemer2}, \eqref{eqemer3} and \eqref{eq14} we obtain
\begin{align*}
\int_{\Omega}{|\nabla u_{m,n}|}^{p-2}\nabla u_{m,n}\cdot(\nabla u_{m,n}-\nabla u_n)dx \rightarrow 0\; \mathrm{as}\;m \rightarrow \infty.
\end{align*}
This implies that $(u_{m,n})$ converges strongly to $u_n$ in $W_{0}^{1,p}(\Omega)$. Thus we have proved that the functional $I_n$ satisfies the Cerami condition.
\end{proof}
\noindent By the lemmas in Lemma $\ref{lem1}$, $\ref{lem2}$, $\ref{lem3}$, we can conclude that there exists $\lambda'$ such that for every $\lambda\in (0,\lambda')$ the functional $I_n$ satisfies the assumption of the Mountain-Pass theorem \cite{Evans}. Hence, there exists critical point $u_{n}\in \{u\in W_0^{1,p}(\Omega):\|u\|_{p^{*}}=1\}$ corresponding to each $\mu_{n}$ such that $I_n(u_n)=c>0$. So, $u_{n}$ will satisfy its weak formulation, i.e.
$\left\langle I_n^{\prime}(u_{n}),v\right\rangle  = 0 \;\forall\; v \in W_0^{1,p}(\Omega)$. This implies
\begin{equation}
\int_{\Omega}{|\nabla u_{n}|}^{p-2}\nabla u_{n}\nabla v\;dx-\lambda\int_{\Omega}{|u_{n}|}^{q-2}u_{n}v\;dx-\int_{\Omega}f(x,u_{n})v-\int_{\Omega}\mu_{n}v\;dx = 0
\end{equation}
{\it Proof of Theorem $\ref{Mainthm}$}.~We will now show that there exists another distinct nontrivial solution of the problem using the Ekeland's variational method. Since $I_n$ is a $C^1$ functional hence it is bounded below on the ball $\bar{B}_{r}(0)$. We can thus apply Ekeland variational principle (refer definition \ref{defn3}). Applying this principle to $I_n: {\bar{B}}_{r}(0) \rightarrow \mathbb{R}$ we find that to each $\delta>0$ there exists $u_{\delta} \in \bar{B}_{r}(0)$ such that $I_n(u_{\delta}) < \underset{u \in \bar{B_{r}(0)}}{\text{inf}}I_n(u)+\delta$ and $I_n(u_{\delta}) < I_n(u)+\delta{\|u-u_{\delta}\|}, u \neq u_{\delta}$. From Lemma \ref{lem1} and \ref{lem3}, we know that $\underset{u \in \partial  B_{r}(0)}{\text{inf}} I_n(u) \geq M > 0$ and $\underset{u \in \bar{B}_{r}(0)}{\text{inf}} I_n(u)< 0$.\\
Choose $\delta > 0$ such that $0 < \delta < \underset{u \in \partial B_{r}(0)}{\text{inf}} I_n(u)-\underset{u \in {\bar{B}}_{r}(0)}{\text{inf}} I_n(u)$.\\
Hence, $I_n(u_{\delta}) < \underset{u \in \partial B_{r}(0)}{\text{inf}} I_n(u)$ and so by the choice of $u_{\delta}$ we have $u_{\delta} \in B_{r}(0)$. We define another functional $J_n : {\bar{B}}_{r}(0) \rightarrow \mathbb{R}$ by $J_n(u)= I_n(u)+\delta{\|u-u_{\delta}\|}$. Due to the Ekeland variational principle, Definition \ref{defn3}, we have that $u_{\delta}$ is a minimum point of $J_n$. So 
$\frac{J_n\left(u_{\delta}+t\phi\right)-J_n(u_{\delta})}{t} \geq 0\; \forall\; t > 0$ small and \;$\forall\; \phi \in B_{r}(0)$.\\
Hence, $\frac{I_n\left(u_{\delta}+t\phi\right)-I_n(u_{\delta})}{t}+\delta{\|\phi \|} \geq 0$ and
$\left\langle I_n^{\prime}(u_{\delta}),-\phi\right\rangle   \geq -\delta{\|\phi \|}$ as $t \rightarrow 0^{+}$. Since, $-\phi \in B_{r}(0)$ we replace $\phi$ with $-\phi$ to get 
$\left\langle I_n^{\prime}(u_{\delta}),-\phi\right\rangle\geq -\epsilon{\|\phi \|}$ and hence $\left\langle I_n^{\prime}(u_{\delta}),-\phi\right\rangle   \leq \epsilon{\|\phi \|}$. This implies that ${\| I_n^{\prime}\left(u_{\delta}\right)\|} \leq \delta$.\\
Therefore there exists a sequence $(w_{m,n}) \subset B_{r}(0)$ such that $$I_n(w_{m,n}) \rightarrow {c^{\prime}}=\underset{u \in {\bar{B}}_{r}(0)}{\text{inf}} I_n(w_n)< 0$$ and $I_n^{\prime}(w_{m,n}) \rightarrow 0$ in $W_{0}^{1,p}(\Omega)$ as $m \rightarrow \infty$.
From Lemma $\ref{lem4}$, the sequence $(w_{m,n}) \rightarrow v_{n}$ in $W_{0}^{1,p}(\Omega)$ as $m \rightarrow \infty$. \\
Hence, $I_n(v_{n})=c^{\prime}, I_n^{\prime}(v_{n})=0$.
So, $v_n$ is a non-trivial weak solution of the considered problem.
 Since, $I_n(u_{n})=c > 0 > c^{\prime}=I_n(v_{n})$ so $u_{n}$ and $v_{n}$ are distinct nontrivial solutions of the problem $(P_n)$. Hence, the Theorem $\ref{Mainthm}$ is proved for the sequence of problems $(P_n)$.\\Choose a test function $v=T_k(u_n)$, where $T_k$ is a truncation operator defined as \[T_k(t)=\begin{cases}
 t, & |t| < k \\
 k, & |t |\geq k.
 \end{cases}\]
 Clearly $T_k(u_{n})\in W_0^{1,p}(\Omega)$. Define $A=\{x:|u_n(x)|\geq k\}$. We have 
 \begin{align*}
 \{|\nabla u_n|> t\} & = \{|\nabla u_n|> t,|u_n|< k\} \cup \{|\nabla u_n| > t,|u_n| \geq k\}
 \\& \subset \{|\nabla u_n|> t,|u_n|< k\} \cup \{|u_n| \geq k\}\subset \Omega.
 \end{align*}
 Hence, by the subadditivity of Lebesgue measure, we have
 \begin{equation}\label{eq5}
 |\{|\nabla u_n|> t\}| \leq |\{|\nabla u_n|> t,|u_n|< k\}| + |\{|u_n| \geq k\}|.
 \end{equation}
 Hence we have
 \begin{align*}
 \int_\Omega |\nabla T_k(u_n)|^p & \leq \lambda\int_\Omega |u_n|^{q-2}u_n T_k(u_n)+\int_{\Omega}f(x,u_n)T_k(u_n)+\int_{\Omega}\mu_nT_k(u_n)
 \\& \leq k\lambda|\Omega|^{1/q}\|u_n\|_{q}^{q/q'}+\epsilon\int _{(|u_n| > T)}|u_n|^{p^{*}-1}T_k(u_n)+\int _{\Omega\times[-T,T]}f(x,u_n)T_k(u_n)
 \\& ~~~+\int_{\Omega}\mu_nT_k(u_n)
 \\&  \leq C_1(\lambda,q,\Omega)k+C_2(\epsilon,\Omega)k+k\int_{\Omega}\mu_n
 \\& \leq Ck,
 \end{align*}
 where we have used the condition $(f_0)$ to bound the second integral and the $L^1$ bound of the sequence $(\mu_n)$, due to $\mu_n\rightharpoonup \mu$, to bound the third integral. Restricting the above integral on $A_1={\left\lbrace x: |u_n| < k \right\rbrace}$ we get,
 \begin{align}
 \int_{\left\lbrace |u_n| < k \right\rbrace} |\nabla T_k(u_n)|^p \nonumber& = \int_{\left\lbrace |u_n| < k \right\rbrace} |\nabla u_n|^p\nonumber  \\& \geq \int_{\left\lbrace |\nabla u_n|>t, |u_n|< k \right\rbrace} |\nabla u_n|^p\nonumber \\&\geq t^p|(\{|\nabla u_n|> t,|u_n| < k\}|\nonumber
 \end{align}
 so that, $$|\{|\nabla u_n|> t,|u_n| < k\}|\leq \frac{Ck}{t^p} \hspace{0.4cm}\forall k \geq 1.$$
 Therefore, from the Sobolev inequality $$\lambda_1\Bigg(\int_\Omega |T_k(u_n)|^{p^*}\Bigg)^{\frac{p}{p^*}}\leq \int_{\Omega}|\nabla T_k(u_n)|^p \leq Ck,$$ where, $\lambda_1$ is the first eigen value of the $p$-laplacian operator.
 Now, if we restrict the integral on the left hand side on $A_2=\{x: |u_n(x)| \geq k\} $, on which $T_k(u_n)=k$, we then obtain $$k^p|\{|u_n|\geq k\}|^{\frac{p}{p^*}}\leq Ck,$$
 so that $$|\{|u_n|\geq k\}|\leq \frac{C}{k^{\frac{N(p-1)}{N-p}}} \hspace{0.4cm} \forall k\geq1.$$ So, $(u_n)$ is bounded in $M^{\frac{N(p-1)}{N-p}}(\Omega)$. Now \eqref{eq5} becomes
 \begin{eqnarray}
 |\{|\nabla u_n|> t\}| &\leq & |\{|\nabla u_n|> t,|u_n| < k\}| + |\{|u_n| \geq k\}|\nonumber\\
 &\leq& \frac{Ck}{t^p} + \frac{C}{k^{\frac{N(p-1)}{N-p}}}, \forall k>1.\nonumber
 \end{eqnarray}
 We then choose $k=t^{\frac{N-p}{N-1}}$ and we get $$ |\{|\nabla u_n|> t\}|\leq \frac{C}{t^{\frac{N(p-1)}{N-1}}} \hspace{0.2cm} \forall t\geq 1,$$
 We have thus shown that $(\nabla u_n)$ is bounded in $M^{\frac{N(p-1)}{N-1}}$ and hence bounded in $W_0^{1,s}(\Omega)$ for $s<\frac{N(p-1)}{N-1}$.\\
 From the Definition $\ref{defn2}$ we have $\mu_n \rightharpoonup \mu$. Since ($u_n$) is bounded in $W_0^{1,s}(\Omega)$, which is a reflexive space, we have a subsequence such that $u_n\rightharpoonup u$ in $W_0^{1,s}(\Omega)$. Since $p-1<\frac{N(p-1)}{N-1}$ always holds and according to the assumption $q<p-1$, we have that $u_n \rightarrow u$ in $L^q(\Omega)$ by Rellich's compact embedding. Further, since $u_n \rightarrow u$ in $L^q(\Omega)$, hence by the Egoroff's theorem there exists a subsequence such that $u_n(x) \rightarrow u(x)$ almost everywhere in  $\Omega$. Thus, by the continuity of $f$ we have $f(x,u_n(x))\rightarrow f(x,u(x))$ in $\Omega$ almost everywhere. Summing up these results we find that 
 \begin{align}
 \int_{\Omega}\lambda|u_n|^{q-2}u_n.v &\rightarrow \int_{\Omega}\lambda|u|^{q-2}uv\nonumber\\
 \int_{\Omega}f(x,u_n)v&\rightarrow \int_{\Omega}f(x,u)v \nonumber\\
 \int_{\Omega}\mu_n v &\rightarrow\int_{\Omega}v d\mu
 \end{align}
 $\forall v\in W^{1,p}(\Omega)\cap C_0(\bar{\Omega})$. Since $\lambda|u|^{q-2}u+f(x,u)+\mu=\mu_{u}$, say, is a bounded Radon measure, we look at the following problem 
 \begin{align}
 -\Delta_p z&=\mu_u,~\text{in}~\Omega\nonumber\\
 v&=0,~\text{on}~\partial\Omega.\label{bocc_prob}
 \end{align}
 From \cite{bocc}, there exists a solution to \eqref{bocc_prob}. It can be guaranteed (refer Appendix A) that the sequence $(u_n)$ is compact in $W_0^{1,s}(\Omega)$, for $s\in \left[q,\frac{N(p-1)}{N-1}\right)$. In the course of the proof we have shown that $(\nabla u_n(x))$ is a Cauchy sequence over $\Omega$  (Claim $\ref{claim1}$ in Appendix A). 
 \begin{align}
 \int_{\Omega}|\nabla u|^{p-2}\nabla u.\nabla v&=\lim_{n\rightarrow\infty}\int_{\Omega}|\nabla u_n|^{p-2}\nabla u_n.\nabla v\nonumber\\
  &= \int_{\Omega}|\nabla z|^{p-2}\nabla z.\nabla v\nonumber
 \end{align}
$\forall v\in C_0^{\infty}(\bar{\Omega})$. Thus $z=u$ and hence the given problem has a solution. The argument can be repeated for the sequence of solution $(v_n)$ to produce a nontrivial solution, say, $w$. Thus we have shown the existence of two nontrivial solutions to the problem in \eqref{main_prob}.
\section{Appendix A}
The proof is motivated from the Lemma 1 in \cite{bocc}.\\
$(u_n)$ is bounded in $W_0^{1,s}(\Omega)$, for $s\in\left[q,\frac{N(p-1)}{N-1}\right)$ which implies that there exists a subsequence, which we will still denote as $(u_n)$, $u_n\rightharpoonup u$. This further implies that $u_n\rightarrow u$ in $L^s(\Omega)$ and hence there exists a subsequence such that $u_n(x)\rightarrow u(x)$ almost everywhere in $\Omega$. This further implies that $(u_n)$, $(\nabla u_n)$ is bounded in $L^1(\Omega)$.\\
\begin{claim}
$\nabla u_n(x)\rightarrow \nabla u(x)$ in $\Omega$\label{claim1}
\end{claim}
\noindent Given $\eta>0$, $\epsilon>0$ and set $B>1$, $k>0$ $(n,m\in\mathbb{N})$. Define the following measurable sets.
\begin{align}
E_1&=\{x\in\Omega:|\nabla u_n(x)|>B\}\cup\{x\in\Omega:|\nabla u_m(x)|>B\}\cup\{x\in\Omega:|u_n(x)|>B\}\nonumber\\
&~~~~\cup\{x\in\Omega:|u_m(x)|>B\}\nonumber\\
E_2&=\{x\in\Omega:|u_n(x)-u_m(x)|>k\}\nonumber\\
E_3&=\{x\in\Omega:|\nabla u_n(x)|\leq k,|\nabla u_m(x)|\leq k,|u_n(x)|\leq k,|u_m|\leq k,|u_n(x)-u_m(x)|\leq k,\nonumber\\
&~~~~|\nabla(u_n-u_m)(x)|\geq \eta\}.\nonumber
\end{align}
We remark here that $\{x\in\Omega:|\nabla(u_n-u_m)|\geq \eta\}\subset E_1\cup E_2\cup E_3$. Since $(u_n)$, $(\nabla u_n)$ is bounded in  $L^1(\Omega)$, hence we can choose a sufficiently large $B$, independent of $n,m$, such that $|E_1|<\epsilon$. By the inequality \cite{grey}
\begin{align}
(|X|^{p-2}X-|Y|^{p-2}Y).(X-Y)&\geq C_p |X-Y|^p,~\text{if}~p\geq 2\nonumber\\
&\geq C_p\frac{|X-Y|^2}{(|X|+|Y|)^{2-p}} ~\text{if}~1<p< 2\nonumber
\end{align} 
we have $(|\nabla u|^{p-2}\nabla u-|\nabla v|^{p-2}\nabla v)>0$. So there exists a measurable function $\gamma$ such that $(|\nabla u|^{p-2}\nabla u-|\nabla v|^{p-2}\nabla v).(u-v)\geq \gamma(x)>0$. Let $a(x,s,\xi)=|\xi|^{p-2}\xi$. Thus we have $[a(x,s,\xi)-a(x,s,\psi)].(\xi-\psi)\geq\gamma(x)$, $\forall s\in\mathbb{R}$, $\xi, \psi\in\mathbb{R}^{N}$ such that $|s|, |\xi|, |\psi|\leq B$, $|\xi-\psi|\geq\eta$, $\forall x\in\Omega$.\\
In fact there exists a set $C\subset\Omega$ such that $|C|=0$ and $a(x,s,\xi)$ is continuous over $\Omega\setminus C$. Define
\begin{align}
K&=\{(s,\xi,\psi)\in\mathbb{R}^{2N+1}:|s|\leq B,|\xi|\leq B,|\psi|\leq B, |\xi-\psi|\geq \eta\}\nonumber
\end{align} 
which is a compact set in $\mathbb{R}^{2N+1}$. Then 
\begin{align}\label{app_eq1}
\inf\{[a(x,s,\xi)-a(x,s,\psi)].(\xi-\psi):(s,\xi,\psi)\in K\}&=\gamma(x)>0.
\end{align}
by the compactness of $K$. Thus by \eqref{app_eq1} we have
\begin{align}\label{app_eq2}
\int_{E_3}\gamma&\leq \int_{E_3}[a(x,u_n,\nabla u_n)-a(x,u_n,\nabla u_m)].\nabla(u_n-u_m)\nonumber\\
&\leq \int_{E_3}[a(x,u_m,\nabla u_m)-a(x,u_n,\nabla u_m)].\nabla(u_n-u_m)\nonumber\\
&~~~~+\int_{E_3}[a(x,u_n,\nabla u_n)-a(x,u_m,\nabla u_m)].\nabla(u_n-u_m)\nonumber\\
&\leq \int_{E_3}[a(x,u_m,\nabla u_m)-a(x,u_n,\nabla u_m)].\nabla(u_n-u_m)+2kM
\end{align}
where $M$ is the $L^1$ bound of the sequence of measures $(|u_n|^{q-2}u_n+f(x,u_n)+\mu_n)$. Due to the continuity of $a(x,s,\xi)$ with respect to $(s,\xi)$ almost everywhere in $\Omega$ we thus have for each $\bar{\epsilon}>0$ $\exists~\delta(x,\bar{\epsilon})\geq 0$ (with $|\{x\in\Omega:\delta(x,\bar{\epsilon})=0\}|=0$) such that $|s-s'|\leq\delta(x,\bar{\epsilon})$, $|s|, |s'|, |\xi|\leq B$ which implies $|a(x,s,\xi)-a(x,s,\psi)|\leq\bar{\epsilon}$.
We remark here that $\underset{k\rightarrow 0}{\lim}|\{x\in\Omega:\delta(x,\bar{\epsilon})|=0$.\\
Let $\delta>0$ be from Lemma 2 of \cite{bocc} which is as follows.
\begin{lemma}\label{bocc_lemma}
Let $(X,\Sigma,|.|)$ be a measurable space such that $|X|<\infty$. Let $f:X\rightarrow[0,\infty]$ such that  $|\{x\in X:f(x)=0\}|=0$. Then for any $\epsilon>0$ $\exists ~ \delta>0$ such that $\int_{A}fdm\leq\delta$ implies $|A|\leq\epsilon$.
\end{lemma}
We now choose $\bar{\epsilon}$ such that $\bar{\epsilon}<\frac{\delta}{3}$ and $k>0$ such that $|E_3\cap\{x:\delta(x,\bar{\epsilon})<k\}|<\frac{\delta}{3}$ and $2kM<\frac{\delta}{3}$. Then we finally have 
\begin{align}
\int_{E_3}\gamma&<\delta.
\end{align}
Thus from the Lemma 2 in \cite{bocc} (stated above in $(\ref{bocc_lemma})$) we have $|E_3|<\epsilon$ independent of $n,m$. This guarantess our claim that $(\nabla u_n)$ is a Cauchy sequence in $\Omega$. Now since $(\nabla u_n)$ is a Cauchy sequence in $\Omega$ we have that $\nabla u_n(x)\rightarrow v(x)$ almost everywhere in $\Omega$. Therefore, it is also Cauchy in $W_0^{1,2}$ topology and hence convergent to, $v$, say. In addition to this, we also have that $(u_n)$ is weakly convergent in $W_0^{1,s}(\Omega)$ for $s\in\left[q,\frac{N(p-1)}{N-1}\right)$ to $u$, i.e. $\nabla u_n\rightharpoonup \nabla u$. These two results implies that $v=\nabla u$ and
therefore, $(u_n)$ is compact and $\nabla u_n\rightarrow \nabla u$ in $W_0^{1,s}(\Omega)$. 
\section*{Acknowledgement}
The author Amita Soni thanks the Department of
Science and Technology (D. S. T), Govt. of India for financial
support. Both the authors also acknowledge the facilities received
from the Department of mathematics, National Institute of Technology
Rourkela.

{\sc Amita Soni} and {\sc D. Choudhuri}\\
Department of Mathematics,\\
National Institute of Technology Rourkela, Rourkela - 769008,
India\\
e-mails: soniamita72@gmail.com and dc.iit12@gmail.com.

\begin{thebibliography}{00}
\bibitem{Ambro}
A.Ambrosetti and P.H.Rabinowitz, Dual Variational methods in critical point theory and applications, {\it J. Func. Anal.}, 14 (1973), 349-381.
\bibitem{Narici}
{\sc G.Bachman and L.Narici}, {\it Functional Analysis}(Dover Publications, Mineola, New York, 1966).
\bibitem{Bisci}
G. Molica Bisci, D.Repov\v{s} and R. Servadei, Nontrivial solutions of superlinear nonlocal problems, Forum Math, 28:6 (2016), 1095-1110 .
\bibitem{bocc}L. Boccardo, T. Gallou\"{e}t, Nonlinear elliptic equations with right hand side measures, Commun. in Partial Differenatial Equations, 17 (3 \& 4) (1992), 641-655.
\bibitem{Ponce}
H.Brezis, M.Marcus and A.Ponce, Nonlinear elliptic equations with measures revisited, (arXiv:1312.6495 [math.AP])
\bibitem{Chung}
N.T.Chung, P.H.Minh and T.H.Nga, Multiple solutions for p-Laplacian problems involving general subcritical growth in bounded domains, {\it Elec. Jour. of Diff. Eqn.}, 78 (2016), 1-12.
\bibitem{Costa}
D.G.Costa and C.A.Magalh\~aes, Variational elliptic problems which are nonquadratic at infinity, {\it Nonlin. anal.}, 23 (1994), 1401-1412.
\bibitem{Ekeland}
I.Ekeland, On the variational principle, {\it J. Math. Anal. Appl.}, 47(1974), 324-353.
\bibitem{grey} G. Ercole and G. A. Pereira, Fractional Sobolev inequalities associated with singular problems, {\it Mathematische Nachrichten}, 10.1002/mana.201700302.
\bibitem{Evans}
{\sc L.C. Evans}, {\it Partial Differential Equations} (Amer. Math.
Soc., 2009)
\bibitem{Benilan} P. B\'{e}nilan, L. Boccardo, T. Gallou\"{e}t, R. Gariepy, M. Pierre, J.L. Vazquez, An $L^1$ theory of existence and uniqueness of solutions of nonlinear elliptic equations, Ann. Scuola Norm. Sup. Pisa, 22, (1995), 241-273.
\bibitem{Bin}
B.Ge, Q.Zhou and L.Zu, Positive solutions for nonlinear elliptic problems of {\it{p}}-Laplacian type on $\mathbb{R}^{N}$ without (AR) condition, {\it Nonlinear Anal.}, 21(2015), 99-109.
\bibitem{Giri}
R.K.Giri and D.Choudhuri, A problem involving the {\it{p}} Laplacian operator, {\it Differential Equations \& Applications}, 9(2)(2017), 171-181.
\bibitem{ittu1}
L. Iturriaga, S. Lorca, P. Ubilla, A quasilinear problem without the Ambrosetti-Rabinowitz type condition, {\it Proc. Roy. Soc. Edinburgh Sect. A}, 140 (2010), 391-398.
\bibitem{kesavan} {\sc Kesavan}, {\it Topics in Functional Analysis and applications} (New age international pvt. ltd., 2003).
\bibitem{kris1}
A. Krist\'{a}ly, H.Lisei, C. Varga, Multiple solutions for p-Laplacian type equations, {\it Nonlinear Anal. (TMA)}, 68 (2008), 1375-1381.
\bibitem{lan1} Y.Y. Lan, Existence of solutions to p-Laplacian equations involving general subcritical growth, {\it Electronic J. of Diff. Equ.}, 2014 (151) (2014), 1-9.
\bibitem{lan2} Y.Y. Lan, C.L. Tang; Existence of solutions to a class of semilinear elliptic equations
involving general subcritical growth, {\it Proc. Roy. Soc. Edinburgh Sect. A}, 144 (2014), 809-818.
\bibitem{Liu}
S.Liu, On superlinear problems without Ambrosetti-Rabinowitz condition, {\it Nonlinear Anal.}, 73(3), (2010), 788-795.
\bibitem{li1} G. Li, C. Yang, The existence of a nontrivial solution to a nonlinear elliptic boundary value problem of p-Laplacian type without the Ambrosetti-Rabinowitz condition, {\it Nonlinear Anal.
(TMA)}, 72 (2010), 4602-4613.
\bibitem{Miyagaki}
O.H.Miyagaki and M.A.Souto, Superlinear problems without Ambrosetti and Rabinowitz growth condition, {\it J. Diff. Eqns}, 245(2008), 3628-3638.
\bibitem{Martin}
M.Schechter, The Use of Cerami Sequences in Critical Point Theory, {\it Abstract and Applied Analysis}, Article ID 58948, 28 pages, 2007.
\bibitem{sun1}
M. Z. Sun, Multiple solutions of a superlinear p-Laplacian equation without AR-condition, {\it Appl. Anal.}, 89 (2010), 325-336.
\bibitem{wang1}
J. Wang, C. L. Tang, Existence and multiplicity of solutions for a class of superlinear $p$-Laplacian equations, {\it Boundary Value Problems}, {\bf 2006}, (2006), 1-12, Art ID 47275.
\bibitem{Xiang}
M.Xiang and B.Zhang, Degenerate Kirchhoff problems involving the fractional {\it{p}}-Laplacian without the (AR) condition, {\it Comp. Var. and ellip. eqn.}, 60(9), 2015, 1277-1287.
\end{thebibliography}
\end{document}